\documentclass{proc-lm}
\usepackage{graphicx}
\usepackage{color}
\usepackage{float}
\newtheorem{theorem}{Theorem}[section]

\theoremstyle{definition}
\newtheorem{definition}[theorem]{Definition}
\newtheorem{example}[theorem]{Example}

\theoremstyle{proposition}
\newtheorem{proposition}[theorem]{Proposition}

\theoremstyle{remark}
\newtheorem{remark}[theorem]{Remark}

\numberwithin{equation}{section}
\def\eps{\varepsilon}

\def\la{\lambda}
\def\La{\Lambda}

\def\w{\omega}

\def\N{\mathbb{N}}
\def\Z{\mathbb{Z}}
\def\R{\mathbb{R}}
\def\C{\mathbb{C}}

\def\T{\mathbb{T}}

\def\Q{\mathbb{Q}}

\def\cA{\mathcal{A}}

\def\cT{{\mathcal T}}

\def\RR{\mathcal{R}}

\newcommand{\peso}[1]{ \ \mbox{ \rm  #1 } \  }

\newcommand{\sub}[2]{{#1}_{\mbox{\tiny{${#2}$}}}}
\begin{document}
% ----------------------------------------------------------------
% ----------------------------------------------------------------
\title{Riesz Bases of Exponentials on Unbounded Multi-tiles}

\author{Carlos Cabrelli}
\address{ Departamento de Matem\'atica, Universidad de Buenos Aires,
Instituto de Matem\'atica "Luis Santal\'o" (IMAS-CONICET-UBA), Buenos Aires, Argentina}
\email{cabrelli@dm.uba.ar}

%    author two information
\author{Diana Carbajal}
\address{ Departamento de Matem\'atica, Universidad de Buenos Aires,
Instituto de Matem\'atica "Luis Santal\'o" (IMAS-CONICET-UBA), Buenos Aires, Argentina}
\email{dcarbajal@dm.uba.ar}

%    \subjclass is required.

\date{}

\dedicatory{}

%    "Communicated by" -- provide editor's name; required.
%\commby{}

\footnotetext{The research of the authors is partially supported by Grants: CONICET PIP 11220110101018, PICT-2014-1480, 
UBACyT 20020130100403BA, UBACyT 20020130100422BA.} 

\begin{abstract}
	
	We prove the existence of Riesz bases of exponentials of $L^2(\Omega)$, provided that $\Omega\subset \R^d$ is a measurable set of finite and positive measure, not necessarily bounded, that satisfies a multi-tiling condition and an arithmetic property that we call {\it admissibility}. This property is satisfied for any bounded domain, so our results extend the known case of bounded multi-tiles. 
	We also extend known results for submulti-tiles and frames of exponentials to the unbounded case.
	
\end{abstract}

\maketitle


%\tableofcontents
%\nocite{*}
% ----------------------------------------------------------------
% ----------------------------------------------------------------
\section{Introduction}
% ----------------------------------------------------------------
% ----------------------------------------------------------------

The main goal of this paper is to study the existence of Riesz basis of exponentials in $L^2(\Omega)$ 
for domains  $\Omega \subset \R^d$ of finite and positive measure, not necessarily bounded.

The existence of bases of exponentials is a very well studied problem. For orthonormal bases, the question of existence is related to the famous Fuglede's conjecture \cite{F} (also known as the spectral set conjecture).
It states that if $\Omega$ is a  domain of positive and finite measure, an orthogonal basis of exponentials $\big\{e^{2\pi i \gamma . \omega} : \gamma \in \Gamma\big\}$ for $L^2(\Omega)$ exists if and only if the set $\Omega$ tiles $\R^d$ by translations along some discrete set $\Lambda$. This latter means that
\begin{equation*}
\sum_{\la\in\La} \chi_{\Omega}(\w+\la) = 1, \;\;\;\text{ a.e.} \;\omega \in \R^d.
\end{equation*}

Fuglede's conjecture is false in dimensions greater or equal than $3$, and it is open for $d=1,2$. (See \cite{T, KM, FMM, FR}). However, it has been proved for a great number of special cases. For example, it is always true for lattices \cite{F}, (see also \cite{Ios}). That is, if $H$  is a 
full lattice in $\R^d$, the system
$\big\{ e^{2\pi i h. \omega} : h \in H\big\}$ is an orthogonal basis of $L^2(\Omega)$ if and only if
$\Omega$ tiles $\R^d$ with translations by $\Lambda$, the dual lattice of $H$.  It is also true for convex bodies, \cite{IKT1}.

On the other hand, it has also been proved that there are sets $\Omega$ that do not possess an orthonormal basis of exponentials, as it is the case, for example, of the unit ball of $\R^d$ when $d > 1$ and the case of non-symmetric convex bodies \cite{Ko1}.
Since orthogonality impose a very severe restriction, it is natural to look at Riesz bases instead. 

The system $\{e^{2\pi i \gamma . \omega} : \gamma \in \Gamma\}$ is a {\bf Riesz basis} of $L^2(\Omega)$  if it is complete and satisfies
that 
\begin{equation*}\label{BaseRiesz}
A \sum_{\gamma \in \Gamma} |\,c_{\gamma}\,|^2 \leq \big\| \sum_{\gamma \in \Gamma} c_{\gamma} e^{2\pi i \gamma . \omega} \big\|^2  \leq B\sum_{\gamma \in \Gamma} |\,c_{\gamma}\,|^2, \quad \forall \, \{c_{\gamma}\} \in \ell^2(\Gamma),
\end{equation*}
for some  positive constants  $A,B>0.$ 

The more general problem of the existence of Riesz bases of exponentials is of different nature and bring new challenges. Again the relevant question here is which domains $\Omega$ admit a Riesz basis of exponentials, and which discrete sets $\Gamma$ give rise to Riesz basis of exponentials for some domain. There are few cases of sets where it was possible to prove the existence of such bases. However, as far as we know, there is no example of 
a set $\Omega$ of finite measure, (even in the line), that do not support a basis of this type.

One of the 
reasons that make the problem significant and relevant is that the existence of a Riesz  basis of exponentials for a set $\Omega$ is equivalent to the existence of a set of stable sampling and interpolation for the associated Paley-Wiener space $PW_{\Omega}$
(see, for example, \cite{Y,S}).

Recently, G. Kozma and S. Nitzan made  a significative advance for this problem. They proved
that any finite union of rectangles in $\R^d$ admits a Riesz basis of exponentials, \cite{KN, KNd}.

Morover, S. Grepstad and N. Lev \cite{GeL}, discovered that bounded measurable sets $\Omega\subset \R^d$ that satisfy a multi-tiling condition,
support a Riesz basis of exponentials. The proof use the theory of quasi-crystals developed in \cite{MM1,MM2}, and require the condition that  the boundary of the domain $\Omega$ has Lebesgue measure zero.
Later on, Kolountzakis \cite{Ko} found a much simpler proof and was able to remove the zero measure boundary condition.
More precisely, they proved that if a bounded measurable set $k$-tiles $\R^d$ by translations on a lattice $\Lambda$ (see Definition \ref{multi-tiles-def}), then there exist
vectors $a_1,\dots,a_k \in \R^d$ such that $E(H, a_1,\dots,a_k)$ is a Riesz basis of $L^2(\Omega).$
Here, 
\begin{equation}\label{special}
E(H\,;\, a_1,\dots,a_k):= \{e^{2\pi i (a_j + \,h) . \omega}\, :\, h\in H,\, j=1,\;\dots,k\}, \end{equation}
where $H$ is the dual lattice of $\Lambda.$
That is, bounded multi-tile sets with respect to a lattice, always  support a basis of exponentials with the set of frequencies being a finite union of translations on the dual lattice.

This result was extended in \cite{AAC} to locally compact abelian groups.  They used fiberization techniques from the theory of shift-invariant spaces, \cite{CP}. They also proved, in this general setting, a converse of this result. That is, if a set $\Omega \subset \R^d$ is such that there exist
a lattice $H$ and vectors $a_1,\dots,a_k \in \R^d$ with $E(H, a_1,\dots,a_k)$ a Riesz basis of $L^2(\Omega),$ then $\Omega$ must multi-tile $\R^d$ at level $k$ for $\Lambda,$ the dual lattice of $H.$
This can be seen as an extension of Fuglede's Theorem for lattices, for the case of multi-tiles and Riesz bases.

A natural question raised by Kolountzakis in \cite{Ko} was if this result was still valid for {\it unbounded} multi-tile sets of finite measure.
In \cite{AAC} the authors answered this question by the negative. They constructed a counterexample of an unbounded multi-tile set of level $2$ in the line, that does not possess a Riesz basis of exponentials with the special structure  \eqref{special}.

In this paper, we prove that {\it unbounded} multi-tile sets of $\R^d$ of finite measure do support a Riesz basis of exponentials
if they satisfy an extra arithmetic condition, that we call {\it admissibility} (see \ref{cP} for a precise definition).
Our main result is:
\begin{theorem}\label{RB for k-tiles}
	Let $\Omega\subset\R^d$ be a measurable set such that $0 < |\Omega|  < +\infty$ and $\La\subset\R^d$ a full lattice.  If
	\begin{enumerate}
		\item[(i)]$\Omega$ multi-tiles $\R^d$ at level $k$ by translations on $\La$,
		\item[(ii)] $\Omega$ is admissible for $\La$,
	\end{enumerate}
	then, there exist $a_1,\dots, a_k \in \R^d$ such that the set $E(H\,;\,a_1,\dots,a_k)$
	is a Riesz basis of $L^2(\Omega)$.
\end{theorem}
In the last section we apply our results to obtain relationships between submulti-tiles (see Definition \ref{k-subtile}) and frames of exponentials.

%{\color{red}
%Let us mention here that most of content of this paper can be extended without difficulty to the setting of locally compact groups.
%We will work in the Euclidean case to make the article more readable.
%}

The paper is organized as it follows: In Section \ref{Pre} we set the notation and introduce the definition of admissibility.
We also review the results from the theory of shift-invariant spaces that we will need later. Section \ref{Mul} is devoted to the proofs of our results on multi-tiles and the existence of Riesz bases of exponentials. Finally, in Section \ref{Sub} we explore the relation between submulti-tiles and frames of exponentials.

\section{Preliminaries}\label{Pre}
% ----------------------------------------------------------------
% ----------------------------------------------------------------

Let $\La\subset\R^d$ be a {\bf full lattice}, this means that there is a $d\times d$ invertible matrix $M$ such that $\La=M\Z^d$. Recall that the {\bf fundamental domain} with respect the lattice $\La$ is the set $D=M\T^d$, which is a set of representatives of the quotient $\R^d/\La$.

Let $H\subset\R^d$ be the {\bf dual lattice} of $\La$, this is the set
\begin{equation*}
H=\{\,h\in\R^d\,:\,\langle h,\la\rangle \in\Z,\text{ for all }\la\in\La\}.
\end{equation*}
It is easy to see that $H=(M^t)^{-1}\Z^d$.

From now on, when working with a lattice $\La$, we will always denote by $D$ its fundamental domain and by $H$ its dual lattice.
For notational simplicity, we will denote by $e_{\alpha}$ the function $e_{\alpha}(\omega) =  e^{2\pi i \alpha\cdot\omega},\; \alpha, \omega \in \R^d,$ and $\#A$ will be the cardinal of the set $A$.

We will also need the following definition.
\begin{definition}\label{struc}
	We will say that a system of exponentials is {\bf structured} if it is of the form $E(H,a_1,\dots,a_k)$ as in \eqref{special} with $H \subset \R^d$ a lattice and $a_1,\dots,a_k \in \R^d$ .
\end{definition}

\subsection{Multi-tiles}

Hereafter, given a set $\Omega\subset\R^d$ and a lattice $\La\subset\R^d$, for every $\w\in D$ we will denote $\La_\w(\Omega)=\La_\w:=\{\,\la\in\La\,:\,\w+\la\in\Omega\,\}$. 

\begin{remark}\label{La_w}
	Observe that if $\Omega\subset\R^d$ is a measurable set of finite measure, then $\La_\w$ must be finite for almost every $\w\in D$. This is because
	\begin{equation*}
	\int_D \sum_{\la\in\La} \chi_\Omega(\w+\la)\,d\w = \int_{\R^d} \chi_\Omega(\w)\,d\w = |\Omega|  <+
	\infty.
	\end{equation*}
\end{remark}

\begin{definition}\label{multi-tiles-def}
	Let $k$ be a positive integer. We say that a measurable set $\Omega\subset\R^d$ {\bf multi-tiles $\R^d$ at level $k$} by translations on a lattice  $\La$ (or that $\Omega$ {\bf $k$-tiles $\R^d$} ) if for almost every $\w\in D$,
	\begin{equation*}\label{k-tile}
	\sum_{\la\in \La} \chi_{\Omega} (\w+\la)=k.
	\end{equation*}
\end{definition}

Notice that if $\Omega$ is a $k$-tile by translations on $\La$, then $\# \La_\w=k$ for almost every $\w \in D$.

\subsection{Admissible sets}

In this subsection we introduce the concept of admissible sets.

\begin{definition}\label{cP}
	Let $\Omega \subset \R^d$  be a finite measure set and $\Lambda$ a full lattice in $\R^d$. We will say that  $\Omega$ is {\bf {admissible for $\Lambda$}}  if there exist a vector $v \in H$ and a number $n\in\N,$
	such that for almost every $\omega \in D$, the numbers $\{\langle v,\lambda\rangle\,:\, \lambda \in \Lambda_{\omega}\}$ are distinct elements $(\text{mod }n)$.
	We will also say in that case that  $\Omega$ is $(n,v)$-admissible for $\Lambda$, if we want to emphasize the dependance on $n$ and $v$.
\end{definition}

When $d=1$ and $\La=\Z$, this is equivalent to say that for almost every $\w\in D$, the elements of $\La_\w\subset\Z$ are all distinct $(\text{mod }n)$.

A graphical way to describe admissibility is the following:  Let $\Omega$ be  admissible
with respect to  $\Lambda$ for some $n\in\N$ and some vector $v \in H$.  Assume that we pick a different color for each of the elements of $\Z_n$,
and we colored $\R^d$ painting the  set  $D + \lambda$ with the color assigned to the remainder $(\text{mod }n)$ of $\langle v,\la\rangle.$
Then  the admissibility says  that for almost all $\omega\in D$  the elements of the form $\omega + \lambda,$ with $\lambda \in \Lambda,$ that belongs to $\Omega$  have different colors!

\begin{remark}
	Every bounded set $\Omega\subset\R^d$ is admissible. This is because in this case, the set $\cup_{\w\in D}\, \La_\w$ must be finite, so for any $v \in H$, one can just choose a number $n\in\N$ large enough for which all the numbers of $\{\langle v,\lambda\rangle\,:\, \lambda \in \Lambda_{\omega}\}$ are all distinct $(\text{mod }n)$. 
\end{remark}

The following example shows that there exist multi-tiles that are not admissible.

\begin{example}\label{not P-adm}
	Consider the partition of $[0,1)$ in intervals $I_j:=\big[\frac{2^j-2}{2^j},\frac{2^j-1}{2^j}\big)$, $j\geq 1$.
	The set
	$$
	\Omega=[0,1)\cup \bigcup_{j=1}^\infty (I_j+j)
	$$
	is an unbounded subset of $\R$ that $2$-tiles by translations on $\Z$ and which is not admissible for $\Z$ (See figure \ref{2-tile}). In order to see that the admissibility fails, note that if $n$ is any fixed natural number and $\w\in I_n$ then $\La_\w=\{0,n\}$, which are not distinct $(\text{mod }n)$. This example is also interesting as this set does not admit a {\it structured} Riesz basis of exponentials for any lattice, see \cite{AAC} for more details.
	
	\begin{figure}[!hp]
		\centering
		\includegraphics[width=\textwidth]{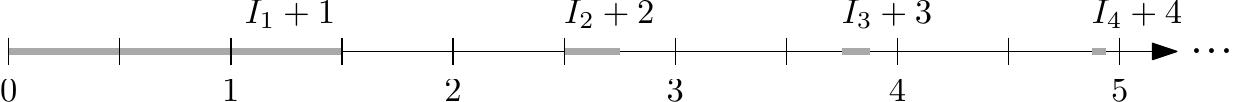}
		\par\vspace{0cm}
		\caption{The set $\Omega$.}
		\label{2-tile}
	\end{figure}
\end{example}

On the other hand, unbounded admissible multi-tiles do exist:
\begin{example}\label{P-adm}
	if in the previous example one translate the intervals $I_j$ only by odd numbers, then
	\begin{equation*}
	\Omega=[0,1)\cup \bigcup_{j=1}^\infty (I_j+2j+1),
	\end{equation*}
	is a $2$-tile unbounded set of $\R$ by translations on $\Z$ that is admissible for $\Z$ taking $n=2$.
\end{example}

\subsection{Shift-invariant spaces}

For the proof of our results we will need to recall some facts from the theory of shift-invariant spaces.
The reader is referred to \cite{B} and \cite{CP}, where a more complete treatment of the general case of shift-invariant spaces in locally compact abelian groups can be found. 

\begin{definition}
	We say that a closed subspace $V\subset L^2(\R^d)$ is {\bf $H$-invariant} if
	$$
	f\in V \peso{then} \tau_h f\in V, \,\,\, \forall h \in H,
	$$
	where $\tau_hf(x)= f(x-h)$.
\end{definition}

{\bf Paley-Wiener} spaces are a family of shift-invariant spaces in which we are especially interested. These spaces are defined by
\begin{equation*}
PW_\Omega = \{\,f\in L^2(\R^d)\,:\,\widehat{f}\in L^2(\Omega)\,\},
\end{equation*}
where $\Omega\subset\R^d$ is a measurable set of finite measure. It is easy to see that, in fact, they are invariant by any translation.

An essential tool in the development of shift-invariant theory is the technique known as {\bf fiberization} that we will introduce now.

\begin{proposition}
	The map $\cT: L^2(\R^d) \to L^2(D,\ell^2(\La))$ defined by
	$$
	\cT f(\omega) = \{\hat{f} (\omega+\la) \}_{\la \in \La},
	$$
	is an isometric isomorphism.
\end{proposition}

The evaluation of elements of $L^2(\R^d)$ could not make sense a priori, however, $\cT$ is a well defined mapping by virtue of the next remark. 

\begin{remark} 
	If $\hat{f}$ and $\hat{g}$ are equal almost everywhere, then for almost every $\omega\in D$,
	$$
	\{\hat{f} (\omega+\la) \}_{\la \in \La}= \{\hat{g} (\omega+\la) \}_{\la \in \La}.
	$$
\end{remark}

In \cite{H}, Helson proved the existence of measurable range functions of an $H$-invariant space $V\subset L^2(\R^d)$. A {\bf range function} is a mapping
\begin{align*}
J_V:D&\rightarrow\{\text{\,closed subspaces of }\ell^2(\La)\,\}\\
\w&\mapsto J_V(\w),
\end{align*}
which has the property that $f\in V$ if and only if for almost every $\w\in D$,
\begin{equation*}
\cT f(\omega)\in J_V(\w).
\end{equation*}
Furthermore, if $V=\overline{\text{span}}\{\,\tau_hf\,:\,h\in H,\,f\in\cA\,\}$ for some countable set $\cA\subset L^2(\R^d)$, then
\begin{equation*}
J_V(\w) = \overline{\text{span}}\{\,\cT f(\omega)\,:\,f\in\cA\,\}.
\end{equation*}

We say $J_V$ is measurable in the following sense: for every $v,w\in\ell^2(\La)$, the scalar function $\w\mapsto\langle P_{J_V(\w)}v,w \rangle$ is measurable, where $P_{J_V(\w)}$ is the orthogonal projection onto $J_V(\w)$. Moreover, a measurable range function of $V$ is essentially unique, i.e., if $V$ has two measurable range functions $J_V$ and $J_V'$, then $J_V=J_V'$ for almost every $\w\in D$. 

Another remarkable result regarding range functions is the characterization of frames and Riesz bases of a shift-invariant space $V$ in terms of the properties of fibers. 

\begin{theorem}
	Let $\cA\subset L^2(\R^d)$ be a countable set. Then, 
	\begin{enumerate}
		\item[(i)] the system $\{\,\tau_h f\,:\,h\in H,\,f\in\cA\,\}$ is a frame of $V$ with constants $A,B>0$ if and only if  $\{\,\cT f(\omega)\,:\,f\in\cA\,\}\subset\ell^2(\La)$ is a frame of $J_V(\w)$ with constants $A,B>0$ for almost every $\w\in D$.
		\item[(ii)] the system $\{\,\tau_h f\,:\,h\in H,\,f\in\cA\,\}$ is a Riesz basis of $V$ with constants $A,B>0$ if and only if  $\{\,\cT f(\omega)\,:\,f\in\cA\,\}\subset\ell^2(\La)$ is a Riesz basis of $J_V(\w)$ with constants $A,B>0$ for almost every $\w\in D$.
	\end{enumerate}
\end{theorem}

In particular, when $\cA\subset L^2(\R^d)$ is a finite set, this allows us to translate problems in infinite dimensional $H$-invariant spaces, into problems of finite dimension that can be treated with linear algebra.

When working with the shift-invariant space $V=PW_\Omega$, we denote its range function as $J_\Omega$. Considering Remark \ref{La_w}, we are able to characterize $J_\Omega$ as it follows.

\begin{proposition}\label{RangeFunc}
	Let $\Omega\subset\R^d$ be a measurable set of finite measure. Then for almost every $\w\in D$ we have
	\begin{equation*}
	J_\Omega(\w)\simeq\ell^2(\La_\w).
	\end{equation*}
\end{proposition}

\begin{proof}
	
	Let us fix $\w\in D\setminus E$ where $E\subset D$ is the zero measure set of exceptions where $\La_\w$ is not finite and define $S_\w:=\{\,a\in\ell^2(\La)\,:\,\text{Supp}(a)\subseteq \La_\w\,\}$, which is isomorphic to $\ell^2(\La_\w)$. Let $C_b(\Omega)$ be the space of bounded continuous functions defined on $\Omega$, we have that $C_b(\Omega)\subset L^2(\Omega)$.
	
	Let $\widetilde{a}\in S_\w$, then there is a sequence $a\in\ell^2(\La_\w)$ such that
	\begin{equation*}
	\widetilde{a}_\la=
	\begin{cases}
	a_\la&\text{ if }\la\in \La_\w,\\
	0&\text{ otherwise.}
	\end{cases}
	\end{equation*}
	
	By Tietze's Extension Theorem, there exists $f_a\in C_b(\Omega)$ such that $f_a(\w+\la)=a_\la$. If we define $\widetilde{f_a}$ as $f_a$ in $\Omega$ and zero in $\R^d\setminus \Omega$, then $(\widetilde{f_a})\check{\text{ }}\in PW_\Omega$, thus $\cT (\widetilde{f_a})\check{\text{ }}(\w)=\widetilde{a} \in J_\Omega(\w)$. This proves that $S_\w\subseteq J_\Omega(\w)$. It is easy to prove the other inclusion. We conclude that $J_\Omega(\w)=S_\w$.
	
\end{proof}

As a consequence, we see that $\Omega$ is a $k$-tile if and only if $J_\Omega(\w)$ are $k$ dimensional for almost every $\w\in D$. 

All these previous results lead to the following theorem whose proof can be found in \cite{AAC}.

\begin{theorem}\label{ShiftBFR0}
	Let $\Omega$ be a $k$-tile measurable subset of $\R^d$. Given $\phi_1,\ldots,\phi_k \in PW_\Omega$ we define
	$$
	T_\w=\begin{pmatrix}
	\widehat{\phi}_1(\w+\la_{1})&\ldots& \widehat{\phi}_k(\w+\la_{1})\\
	\vdots&\ddots&\vdots\\
	\widehat{\phi}_1(\w+\la_{k})&\ldots& \widehat{\phi}_k(\w+\la_{k})
	\end{pmatrix}
	$$ 
	where the $\la_j=\la_j(\w)$ for $j=1,\ldots,k$ are the $k$ values of $\La$ that %$\w+\la\in\Omega$.
	belong to $\La_\w$. Then, the subsequent statements are equivalent:
	
	\begin{itemize}
		\item[(i)] The set $\Phi_{H}=\{\,\tau_{h}\phi_j:\ h\in H,\ j=1,\ldots,k\,\}$ is a Riesz basis for $PW_\Omega$.
		\item[(ii)] There exist $A, B>0$ such that for almost every $\omega \in D$,
		\begin{equation}\label{ineq Tw}
		A||x||^2 \leq \|T_\w \,x\|^2\leq B||x||^2,
		\end{equation}
		for every $x\in \C^k$.
	\end{itemize}
	Moreover, in this case the constants of the Riesz basis are 
	$$
	A=\inf_{\w\in D} \ \|T_\w^{-1}\|^{-1} \peso{and} B=\sup_{\w\in D}\ \|T_\w\|.
	$$
\end{theorem}

We will now show the connection between this theorem and the problem of the existence of Riesz bases of exponentials. 

Let $\Omega\subset\R^d$ be a measurable $k$-tile by translations on a lattice $\La.$ 
We want to find $a_1,\dots,a_k\in\R^d$ such that $E(H\,;\,a_1,\dots,a_k)=\{e_{a_j +\, h} \, :\, h\in H,\, j=1,\dots,k\}$ is a Riesz basis of $L^2(\Omega)$. 

Define $\phi_1,\dots,\phi_k$ by their Fourier transform as follows:
\begin{equation}\label{phi_j}
\hat{\phi}_j:=e_{a_j}\,\sub{\chi}{\Omega},\quad j=1,\dots,k.
\end{equation}
Hence, we are looking for $a_1,\dots,a_k\in\R^d$ such that $\{\hat{\phi}_j e_h \, :\, h\in H,\,  j=1,\dots,k\}$ is a Riesz basis of $L^2(\Omega)$, which is equivalent to $\{\tau_h\phi_j \, :\, h\in H,\,  j=1,\dots,k\}$ being a Riesz basis for $PW_\Omega$.

Theorem \ref{ShiftBFR0} states that this will happen if and only if the matrices
\begin{equation}\label{Tw}
T_\w=\begin{pmatrix}
\widehat{\phi}_1(\w+\la_{1})&\ldots& \widehat{\phi}_k(\w+\la_{1})\\
\vdots&\ddots&\vdots\\
\widehat{\phi}_1(\w+\la_{k})&\ldots& \widehat{\phi}_k(\w+\la_{k})
\end{pmatrix}=\begin{pmatrix}
e_{a_1}\, (\w+\la_1)     &\ldots&       e_{a_k}\, (\w+\la_1)  \\
\vdots&\ddots&\vdots\\
e_{a_1}\, (\w+\la_k)       &\ldots&    e_{a_k}(\w+\la_k) 
\end{pmatrix}
\end{equation}
are uniformly bounded for almost every $\w\in D$.  Note that in this case, the columns of $T_{\omega}$ form a Riesz basis of $\C^k$ for almost every $\omega \in D$ with uniform bounds.

To clarify the relation between (i) and (ii) in Theorem \ref{ShiftBFR0} we sketch the proof from \cite{AAC} adapted to our setting.

The collection $\{e_{a_j +\, h} \, :\, h\in H,\, j=1,\dots,k\}$ is a Riesz sequence in $L^2(\Omega)$ if there exist positive constants $A$ and $B$ such that for any sequence of complex numbers $\{c_{j,h}\}$ with finitely many non-zero terms,
$$ A \;\sum_{j=1}^{k} \sum_{h\in H} |c_{j,h}|^2 \leq \|P\|^2_{L^2(\Omega)} \leq B\; \sum_{j=1}^{k} \sum_{h\in H} |c_{j,h}|^2,$$
where $P$ is the exponential polynomial
$$P(\w)=\sum_{j=1}^{k} \sum_{h\in H} c_{j,h} \;e_{a_j+h}(\w).$$

We see that
\begin{equation*}
\|P\|^2_{L^2(\Omega)}=\int_{\R^d}\bigg| \sum_{j=1}^{k} \sum_{h\in H} c_{j,h} \; e_{a_j+h}(\w)\bigg|^2 \chi_\Omega(\w)\, d\w = \int_{\R^d}\bigg| \sum_{j=1}^{k} m_j(\w)\; e_{a_j}(\w)\bigg|^2 \chi_\Omega(\w)\, d\w,
\end{equation*}
where 
$$m_j(\w):=\sum_{h\in H} c_{j,h} \; e_{h}(\w),\quad j=1,\dots,k.$$
By a $\La$-periodization argument, this is equal to
\begin{equation*}\label{sum todos la}
\sum_{\la\in\La} \int_D \bigg| \sum_{j=1}^{k} m_j(\w) e_{a_j}(\w+\la)\bigg|^2 \chi_\Omega(\w+\la)\, d\w. 
\end{equation*}
Since $\Omega$ $k$-tiles $\R^d$ by translations on $\La$, we have that for almost every $\w\in D$, $\La_\w=\{\la_1(\w),\dots,\la_k(\w)\}$. Therefore we get 
\begin{equation}\label{norma T}
\|P\|^2_{L^2(\Omega)}=\int_D \sum_{l=1}^k \bigg| \sum_{j=1}^{k} m_j(\w) e_{a_j}(\w+\la_l)\bigg|^2 \, d\w = \int_D \|T_\w m(\w)\|^2_{\C^k}\,d\w,
\end{equation}
where $m(\w)=(m_1(\w),\dots,m_k(\w))$ and $T_\w$ is the matrix defined before.

On the other hand, using  that $\big\{\frac{1}{\sqrt{|D|}} \;e_h\,:\,h\in H\big\}$ is an orthonormal basis of $L^2(D)$, we have
\begin{equation}\label{norma m}
\int_D \|m(\w)\|^2_{\C^k}\,d\w = \sum_{j=1}^k \int_D |m_j(\w)|^2 \,d\w = |D| \sum_{j=1}^k\sum_{h\in H} |c_{j,h}|^2.
\end{equation}

Combining \eqref{norma T} and \eqref{norma m} and using standard arguments of measure theory, one may check that $E(H\,;\,a_1,\dots,a_k)$ is a Riesz sequence of $L^2(\Omega)$ if and only if there exist $A,B>0$ such that for almost every $\w\in D$, the inequalities in \eqref{ineq Tw} hold for every $x\in\C^k$.
Actually, inequality \eqref{ineq Tw} implies the completeness in $L^2(\Omega)$ of the system $\{e_{a_j +\, h} \, :\, h\in H,\, j=1,\dots,k\},$ (see \cite{AAC}).
\bigskip

% ----------------------------------------------------------------
% ----------------------------------------------------------------
\section{Multi-tiles and Riesz bases}\label{Mul}
% ----------------------------------------------------------------
% ----------------------------------------------------------------

The proof of Theorem \ref{RB for k-tiles} is based on the techniques used in \cite{AAC}. Without the assumption that $\Omega$ is a bounded domain, we need admissibility as an extra condition. %we will need the following lemma that states that for a prime number $p\in\N$, all the minors of the Fourier matrix of order $p$ are non-zero. See \cite{T2} and the references therein.

%\begin{lemma}\label{Tao's lemma}
% Let $p$ be a prime and $1\leq k \leq p$. Let $x_1,\dots,x_k$ be distinct elements of $\Z_p$, and let $\xi_1,\dots,\xi_k$ also be distinct elements of $\Z_p$. Then the matrix $(e^{2\pi i x_j\xi_r/p})_{1\leq j,r \leq k}$ has non-zero determinant.
%\end{lemma}

\begin{proof}(Of Theorem \ref{RB for k-tiles}).

	As we discussed before, by Theorem \ref{ShiftBFR0}, it suffices to find vectors $a_1,\dots,a_k\in\R^d$ for which there exist $A,B>0$ such that for almost every $\w\in D$, the inequalities in \eqref{ineq Tw} hold for every $x\in\C^k$, where $T_\w$ are the matrices \eqref{Tw}.
	
	Let us note that for every $\w\in D$, the matrix $T_\w$ can be decomposed as
	\begin{align*}&
	\begin{pmatrix}
	e_{a_1}\,(\la_1) &\ldots& e_{a_k}\, (\la_1)   \\
	\vdots&\ddots&\vdots\\
	e_{a_1}\,(\la_{k})&\ldots& e_{a_k}\, (\la_{k})
	\end{pmatrix}\begin{pmatrix}
	e_{a_1}\,(\w)   &0&\ldots&0& 0\\
	0&e_{a_2}\,(\w) &\ldots& 0&0\\
	\vdots&\vdots&\ddots&\vdots&\vdots\\
	0&0&\ldots& e_{a_{k-1}}\,( \w)&0\\
	0&0&\ldots& 0&e_{a_k}\,(\w)
	\end{pmatrix}  = E_\w\, U_\w \,, \label{short}
	\end{align*}
	where $U_\w$ is a unitary matrix. Then, in order to see the inequalities in \eqref{ineq Tw}, it is enough to prove that for almost every $\w \in D$,
	\begin{equation}\label{inequalities Ew}
	A||x||^2 \leq \|E_\w \,x\|^2\leq B||x||^2,
	\end{equation}
	for every $x\in\C^k$.
	
	Since $\Omega$ is admissible for $\La$, there exist $v\in H$ and a number $n\in\N$ such that for almost every $\w\in D$, the elements in $\{\langle v,\lambda\rangle\,:\, \lambda \in \Lambda_{\omega}\}$ are distinct $(\text{mod }n)$. %We define $a_j:=\frac{j-1}{n}v$, $j=1,\dots,k$.
	%Then for any choice of $k$ integers $s_1,\dots,s_k$ distinct $(\text{mod }p)$, define $a_j:=\frac{s_j}{p}v$, $j=1,\dots,k$. 
	
	Set $\mathcal F_n = \big\{e^{2\pi i  rs / n}\big\}_{0\leq r,s\leq n-1}$ to be the Fourier matrix of order $n$.
	Any $k\times k$ submatrix of $\mathcal F_n$, formed by choosing $k$ consecutive columns and any $k$ rows, is an invertible matrix since it is a Vandermonde matrix.
	
	Now, we define $a_j:=\frac{j-1}{n}v$, $j=1,\dots,k$. We obtain that for almost every $\w\in D$, $$E_\w=\big\{e^{2\pi i (j-1)\langle v,\lambda_l\rangle / n}\big\}_{1\leq l,j\leq k}$$
	is one of those submatrices of $\mathcal F_n$ except by a permutation of its rows, and hence invertible.
	
	Moreover, there are finitely many different matrices $E_\w$ because there are finitely many $k\times k$ submatrices of $\mathcal F_n$. Thus, there exist $A,B>0$ such that the inequalities in (\ref{inequalities Ew}) hold for every $x\in\C^k$ and for almost every $\w\in D$.
\end{proof}

\begin{remark}
	The vectors $a_1,\dots,a_k$ defined in the proof of Theorem \ref{RB for k-tiles}, depend only on the vector $v\in H$ and $n\in\N$ from the admissibility condition. Hence, the same structured system of exponentials is a Riesz basis for any  $k$-tile $\Omega$ which is $(n,v)$-admissible for $\La$ .
\end{remark}

\begin{remark}
	If $n$ is a prime number, any selection of $k$ columns and $k$ rows from $\mathcal F_n$ forms an invertible matrix (see \cite{T2} and the references therein). Then, in the proof of Theorem \ref{RB for k-tiles}, if $n$ is a prime number we could also define $a_j:=\frac{s_j}{n}v$, $j=1,\dots,k$ where $s_1,\dots,s_k$ are distinct integers $(\text{mod }n)$. 
	
	In a more general setting, if $n$ is a  power of a prime number, any submatrix of $\mathcal F_n$, formed by any $k$ rows and $k$ columns satisfying that their index set $\{s_1,\dots,s_k\}$ is {\it uniformly distributed over the divisors of $n$}, is invertible (see \cite{ACM} for a definition of uniformly distributed). Thus, in the proof of Theorem \ref{RB for k-tiles}, if $n= p^l$ with $p$ prime and $l$ a positive integer, we might as well define $a_j:=\frac{s_j}{n}v$, $j=1,\dots,k$ where $\{s_1,\dots,s_k\}$ is uniformly distributed over the divisors of $n$.
\end{remark}

It is important to remark that there exist multi-tile sets that admit a {\it structured} Riesz basis of exponentials without being admissible. In the next example we will construct a multi-tile set which is not admissible for $\Z$ but admits a Riesz basis like in \eqref{special}.

\begin{example}
	Let $\{1,a_1,a_2\}$ be linearly independent numbers over $\Q$. Take the partition of $[0,1)$ as in Example \ref{not P-adm} and consider the following $2$-tile set of $\R$ by translations on $\Z$:
	\begin{equation}\label{2-tile example}
	\Omega=[0,1)\,\cup\,\bigcup_{j=1}^{\infty}\,(I_j+n_j),
	\end{equation}
	where the infinite sequence $\{n_j\}_{j\in\N}\subset\N$ will be adequately chosen to fit our purpose (See figure \ref{n_jtile}).
	
	\begin{figure}[H]
		\centering
		\includegraphics[width=\textwidth]{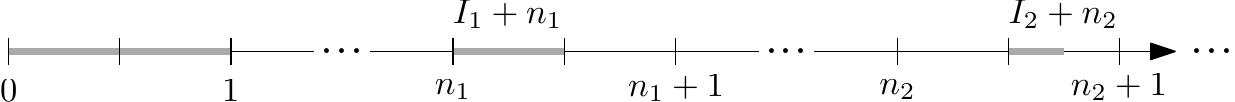}
		\par\vspace{0cm}
		\caption{The set $\Omega$.}
		\label{n_jtile}
	\end{figure} 
	
	Consider the functions $\phi_1$ and $\phi_2$ as defined in (\ref{phi_j}). Recall that in Theorem $\ref{RB for k-tiles}$ we saw that the integer translations of $\phi_1$ and $\phi_2$ form a Riesz basis for $PW_\Omega$ if and only if the matrices
	\begin{equation*}
	E_\w=\begin{pmatrix}
	e_{a_1}\,(\la_1) & e_{a_2}\, (\la_1)   \\
	e_{a_1}\,(\la_{2})& e_{a_2}\, (\la_{2})
	\end{pmatrix}
	=\begin{pmatrix}
	1 & 1   \\
	e_{a_1}\,(\la_{2})& e_{a_2}\, (\la_{2})
	\end{pmatrix}
	\end{equation*}
	satisfy that there exist $A,B>0$ such that (\ref{inequalities Ew}) hold.
	
	Let $\beta_1,\beta_2\in[0,1)$ be two distinct numbers. The matrix,
	\begin{equation*}
	R:=\begin{pmatrix}
	1 & 1   \\
	e^{2\pi i\beta_1}& e^{2\pi i\beta_2}
	\end{pmatrix}
	\end{equation*}
	
	is invertible,   and satisfies that,
	\begin{equation*}
	\gamma_{min}\; ||x||^2 \leq \|R \,x\|^2\leq \gamma_{max} \;||x||^2,\qquad  x \in \R^2,
	\end{equation*}
	where $\gamma_{min}$ and $\gamma_{max}$ are the minimum and maximum eigenvalues of $RR^*$ respectively.
	
	For every $j\in\N$, we have that $\{1,a_1j,a_2j\}$ are also linearly independent over $\Q$. By Kronecker's Approximation Theorem there exists $m_j\in\Z$ for which 
	\begin{equation*}
	\big\|\,(e^{2\pi i a_1 j m_j}\,,\,e^{2\pi i a_2 j m_j})-(e^{2\pi i\beta_1}\,,\,e^{2\pi i\beta_2})\,\big\|_2<\eps.
	\end{equation*}
	Hence, for every $j\in\N$, take $n_j= j m_j$ as the sequence needed in (\ref{2-tile example}). 
	
	Therefore, for almost every $\w\in[0,1)$, the matrices $E_\w E_\w^*$ and $RR^*$ are close to each other. Thus, the eigenvalues of these matrices must be close too. Then, for a small enough $\varepsilon$ we get uniform bounds for (\ref{inequalities Ew}) and consequently $E(\Z\,;\,a_1,a_2)$ is a Riesz basis of $L^2(\Omega)$.
	
	However, this set is not admissible for $\Z$ because for every $j\in\N$, if $\w\in I_j$ then $\La_\w=\{0,j m_j\}$ which are not distinct $(\text{mod }j)$. 
	
\end{example}

\begin{remark}
	A similar argument can be done to extend the previous example to a $k$-tile. If $\{1,a_1,\dots,a_k\}$ are linearly independent numbers over $\Q$, take the multi-tile at level $k$ by translations on $\Z$ set
	\begin{equation*}
	\Omega=[0,k-1)\cup\,\bigcup_{j=1}^{\infty}\,(I_j+n_j), 
	\end{equation*}
	and choose $\{n_j\}_{j\in\N}\subset\N_{\geq k}$ in order to adequately approximate $E_\w$ to an invertible matrix, for almost every $\w\in [0,1)$.
\end{remark}

Hence, a natural question to ask is which sets $\Omega$ support a structured Riesz basis of exponentials. For the {\it bounded} case, it was proved in \cite{AAC} that a set $\Omega\subset\R^d$ which admits a Riesz basis of exponentials $E(H\,;\,a_1,\dots,a_k)$, for some $a_1,\dots,a_k\in\R^d$, must be a $k$-tile of $\R^d$ by translations on $\La$. This result holds true in the case of finite measure sets. 

\begin{theorem}\label{converse RB}
	Let $H\subset\R^d$ be a full lattice and $\La\subset\R^d$ its dual lattice. Given a measurable set of finite measure $\Omega\subset\R^d$, if $L^2(\Omega)$ admits a Riesz basis of the form $E(H\,;\,a_1,\dots,a_k)$
	% \begin{equation}
	%  E(H\,;\,a_1,\dots,a_k) = \{e_{h+a_j} \, \sub{\chi}{\Omega}:\ h\in H\,,\ j=1,\ldots,k\}
	% \end{equation}
	for some $a_1,\dots,a_k\in\R^d$, then $\Omega$ $k$-tiles $\R^d$ by translations on $\La$. 
\end{theorem}

\begin{proof}
	Defining the functions $\phi_j$, $j=1,\dots,k$ as in (\ref{phi_j}), we get that $\{\tau_h{\phi}_j:\ h\in H\,,\ j=1,\ldots,k\}$ is a Riesz basis for $PW_\Omega$. This implies that for almost every $\w\in D$, $\{\cT\phi_1(\w),\dots,\cT\phi_k(\w)\}$ is a Riesz basis of $J_\Omega(\w)$, and hence $\dim J_\Omega(\w) = k$ for almost every $\w\in D$. By proposition \ref{RangeFunc}, we conclude that $\Omega$ is a $k$-tile by translations on $\La$.
\end{proof}

The results  obtained so far can be summarized in Figure \ref{k-tilesgraph}. Note that all the inclusions in the picture are proper.

\begin{figure}[H]
	\centering
	\includegraphics{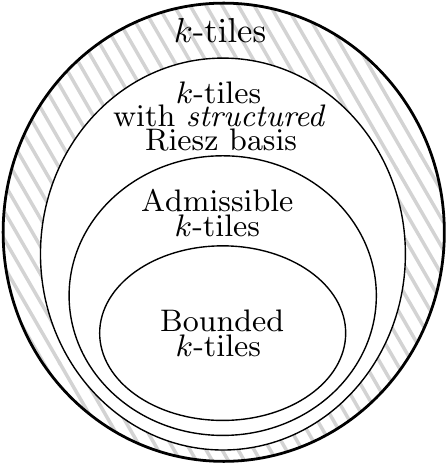}
	\par\vspace{0cm}
	\caption{$k$-tile sets of $\R^d$.}
	\label{k-tilesgraph}
\end{figure}

\section{Submulti-tiles and Frames}\label{Sub}

In this section we turn our attention to frames of exponentials. 
Recall that the system $\{e^{2\pi i \gamma.\w}\,:\,\gamma\in\Gamma\}$ is a {\bf frame} of $L^2(\Omega)$ if it satisfies that
\begin{equation*}
A\|f\|^2\leq \sum_{\gamma\in\Gamma} |\langle f,e^{2\pi i \gamma.\w}\rangle|^2\leq B\|f\|^2,\quad \forall\,f\in L^2(\Omega),
\end{equation*}
for some positive constants $A,B>0$.

It is not difficult to see that any bounded measurable set $\Omega \subset \R^d$ support a frame of exponentials. (This is an easy consequence of these two
facts: $i)$ for any cube $Q$ in $\R^d$ there exists an orthonormal basis of exponentials for $L^2(Q)$. $ii)$ If $\Omega \subseteq Q$, the restriction of an orthonormal basis of exponentials of $L^2(Q)$ to $L^2(\Omega)$ is a frame for the latter.)

Recently S. Nitzan, A. Olevskii and A. Ulanovskii \cite{NOU16} extended the result for any unbounded measurable set of finite measure.
We want to note that the  proof in  \cite{NOU16} used the recently proved Kadison-Singer conjecture and it is not constructive.
The goal of this section is to explore the relationship between unbounded {\it submulti-tiles} and frames and construct concrete  examples of frames of exponentials on unbounded sets.

\begin{definition}\label{k-subtile}
	Let $k$ be a positive integer. We say that a measurable set $\Omega\subset\R^d$ of finite measure, {\bf submulti-tiles $\R^d$ at level $k$} by translations on a lattice $\La$ (or that $\Omega$, {\bf $k$-subtiles $\R^d$}) if for almost every $\w\in D$,
	\begin{equation*}
	\sum_{\la\in \La} \chi_{\Omega} (x+\la) \leq k.\;\;\; {\text{for almost all }} x \in D.
	\end{equation*}
\end{definition}

When $\Omega$ is a $k$-subtile that is admissible, we can no longer claim that $L^2(\Omega)$ has a structured Riesz basis of exponentials, but instead we can see that it supports a structured {\it frame} of exponentials. 
Frames of exponentials are important since they give sets of sampling for the corresponding Paley-Wiener spaces.

The relation between $k$-subtiles and frames of exponentials was first studied in \cite{BHM} for the case when $\Omega$ is a $1$-subtile of finite measure in the context of locally compact abelian groups. Later on, it was proved in \cite{BCHLMM16} that if $\Omega$ is a bounded $k$-subtile, then it admits an structured frame of exponentials. In this section we adapt this last result to the case where $\Omega$ is a finite measure set (not necessarily bounded) with the extra hypothesis of the admissibility. More precisely, we prove the following theorem.

\begin{theorem}\label{frames k-subtiles}
	Let $\Omega\subset\R^d$ be a measurable set such that $0<|\Omega|<+\infty$ and $\La\subset\R^d$ a full lattice. If
	\begin{enumerate}
		\item[(i)]$\Omega$ submulti-tiles $\R^d$ at level $k$ by translations on $\La$,
		\item[(ii)] $\Omega$ is admissible for $\La$,
	\end{enumerate}
	then, there exist $a_1,\dots, a_k \in \R^d$ such that the set $E(H\,;\,a_1,\dots,a_k)$
	is a frame of $L^2(\Omega)$.
\end{theorem}

The strategy of the proof in \cite{BCHLMM16}, consists in, given a bounded $k$-subtile $\Omega$, enlarge it to obtain
a $k$-tile $\Delta$, and then select a structured Riesz basis of $L^2(\Delta),$ (that always exists in the bounded case for $k$-tiles). This basis, when restricted to $\Omega$ is a structured frame for $L^2(\Omega).$

In our case, since the $k$-subtile $\Omega$ is not necessarily bounded, we need to enlarge it to an admissible $k$-tile, to guarantee the existence of the Riesz basis. This requires an adaptation of the proof in \cite{BCHLMM16}. This is done in the next proposition:

\begin{proposition}\label{k-tile embed}
	Let $\Omega$ be a measurable set of finite measure  that $k$-subtiles $\R^d$ and is admissible for a lattice $\La\subset\R^d$. Then, there exists a measurable set $\Delta$, of finite measure, which is a $k$-tile of $\R^d$ and admissible for $\La$ such that $\Omega\subset\Delta$.
\end{proposition}

\begin{proof}
	We start by giving a characterization of sets that $k$-subtile $\R^d$ and are admissibles.
	Let $\La$ be a full lattice in $\R^d,$  $v $ a non-zero vector in the dual lattice $H$ and $n$ a natural number.
	Consider the sub-lattice of $\La$ defined by 
	$$\La^{(0)}:=\{\la\in\La: \; \langle v, \la \rangle \equiv 0 \;(\text{mod }n)\},$$
	and let $\La^{(r)}$, $r=0,\dots,n-1$, be the different cosets of the quotient $\La/\La^{(0)}.$ 
	Let $k \geq 1$ be an integer and $\Omega \subset \R^d$  a $k$-subtile that is $(n,v)-$admissible for $\La.$
	
	Define 
	$$\RR:= \{R \subset \La: \#R\leq k \text{ and } \la-\la' \notin \La^{(0)} \text{ if } \la,\la' \in R, \; \la\neq \la'\}.$$
	The properties imposed on $\Omega$ imply that $\La_{\omega}\in \RR$ for almost every $\omega \in D.$
	
	Now, for $R \in \RR$ set $D_R:= \{\omega \in D: \La_{\omega}=R\}.$
	(Note that if $R \neq R'$, then $D_R \cap D_{R'} = \emptyset$ and that  $D_R$ could be empty for some $R\in \RR$).
	
	We have $D_R + R \subseteq \Omega$ and we obtain (up to measure zero) the decomposition:
	\begin{equation}\label{decomsubomega}
	\Omega =\bigcup_{R\in\RR} D_R +R.
	\end{equation}
	
	We will see now that the sets  $D_R$ are measurables. Consider the functions 
	$$\psi_r(\omega) = \sum_{\la\in \La^{(r)}} \chi_{\Omega} (\omega+\la),\;\;  \omega \in D, \;\; r=0,\dots, n-1,$$
	and let
	$[R]:=\big \{r \in \{0,\dots, n-1\}: r \equiv \langle v, \la \rangle\; (\text{mod }n), \text{ for some  } \la \in R \big \}.$
	
	Thus, $$D_R = \bigcap_{r \in  [R]}  \psi_r^{-1}(1) \; \cap\; \bigcap_{r \notin  [R]} \psi_r^{-1}(0),  $$
	which is an intersection of measurable sets. 
	
	Conversely, for each partition $\big \{D_R: R \in \RR \big \}$ of $D$, in measurable sets (we allow here some of the partition elements
	to have measure zero), the set $\Omega$ defined by \eqref{decomsubomega}, necessarily $k$-subtiles $\R^d$ and is $(n,v)-$admissible for $\La$.
	
	Now that we obtained the desired decomposition,  the proposition follows defining
	\begin{equation*}\label{decomsub}
	\Delta =\bigcup_{R\in\RR} D_R + (R \cup R').
	\end{equation*}
	where for each $R \in \RR$ we have chosen  a set $R' \subseteq  \La$ complementary to $R$, in the sense that
	$[R] \cap [R'] = \emptyset$ and $\# \big( [R] \cup [R'] \big)= k$.
	
\end{proof}

We are now ready to prove Theorem \ref{frames k-subtiles}.

\begin{proof}(of Theorem \ref{frames k-subtiles})
	By Proposition \ref{k-tile embed}, there exists a measurable set of finite measure $\Delta$, such that $k$-tiles $\R^d$ and is admissible for $\La$, which contains $\Omega$. Then, by Theorem \ref{RB for k-tiles} we know that there exist vectors $a_1,\dots,a_k\in\R^d$ such that $E(H\,;\,a_1,\dots,a_k)$ is a Riesz basis of $L^2(\Delta)$. Hence, $E(H\,;\,a_1,\dots,a_k)$ is a frame of $L^2(\Omega)$.
\end{proof}

As we saw in the previous section, Theorem \ref{converse RB} states that if $\Omega$ supports a Riesz basis of exponentials $E(H\,;\,a_1,\dots,a_k)$, for some $a_1,\dots,a_k\in\R^d$, then it must $k$-tile $\R^d$ by translations on $\La$. When $E(H\,;\,a_1,\dots,a_k)$ is a frame instead, a similar result can be proved.

\begin{theorem}
	Let $H\subset\R^d$ be a full lattice and $\La\subset\R^d$ its dual lattice. Given a measurable set of finite measure $\Omega\subset\R^d$, if $L^2(\Omega)$ admits a frame of the form $E(H\,;\,a_1,\dots,a_k)$
	% \begin{equation}
	%  E(H\,;\,a_1,\dots,a_k) = \{e_{h+a_j} \, \sub{\chi}{\Omega}:\ h\in H\,,\ j=1,\ldots,k\}
	% \end{equation}
	for some $a_1,\dots,a_k\in\R^d$, then there exists $\ell\leq k$, such that $\Omega$ $\ell$-subtiles $\R^d$ by translations on $\La$. 
\end{theorem}

\begin{proof}
	Proceeding analogously as in Theorem \ref{converse RB}, we see that $\{\tau_h{\phi}_j:\ h\in H\,,\ j=1,\ldots,k\}$ is a frame for $PW_\Omega$. Which implies that $\{\cT\phi_1(\w),\dots,\cT\phi_k(\w)\}$ is a frame of $J_\Omega(\w)$ for almost every $\w\in D$, and thus $\dim (J_\Omega(\w))\leq k$ for almost every $\w\in D$. By Proposition \ref{RangeFunc}, we get that $\# \La_w\leq k$. Hence, if we take
	\begin{equation*}
	\ell:=\underset{\w\in D}{\text{sup ess}} \sum_{\la\in\La} \chi_{\Omega} (\w+\la),
	\end{equation*}
	$\Omega$ is an $\ell$-subtile of $\R^d$ by translations on $\La$.
\end{proof}

%% \medskip
%{\color{red}
%Open Problems:
%Characterize k-tiles that have a structured basis or equivalently the ones that do not have!
%}

% ------------------------------------------------------------------------------------------------------------------------------------------------------------------------------------------------------
% ------------------------------------------------------------------------------------------------------------------------------------------------------------------------------------------------------

\begin{thebibliography}{X}
	
	
	\bibitem{AAC} Agora, E., Antezana, J., Cabrelli, C., Multi-tiling sets, Riesz bases, and sampling near the critical density in LCA groups. Advances in Mathematics (2015), 285, 454-477.
	
	\bibitem{ACM} Alexeev, B., Cahill, J., Mixon, D. G. (2012). Full spark frames. Journal of Fourier Analysis and Applications, 18(6), 1167-1194.
	
	\bibitem{B} Bownik, M., The structure of shift-invariant subspaces of $L^2(\R^n)$, Journal of Functional Analysis 177 (2000), 282-309.
	
	\bibitem{BCHLMM16} Barbieri, D., Cabrelli, C., Hern\'andez, E., Luthy, P.,  Molter, U. and Mosquera, C., Frames of exponentials and sub-multitiles in LCA groups (2016), preprint available at 	arXiv:1710.03176.
	
	\bibitem{BHM} Barbieri, D., Hern\'andez, E., Mayeli, A., Lattice sub-tilings and frames in LCA groups. Comptes Rendus Mathematique, Volume 355, Issue 2, Pages 193-199.
		
	\bibitem{CP}  Cabrelli, C.,  Paternostro, V., Shift-invariant spaces on LCA groups. J. Funct. Anal. 258 (2010), no. 6, 2034-2059.
		
	\bibitem{FMM}  Farkas, B.,  Matolcsi, M.,  Móra, P., On Fuglede's conjecture and the existence of universal spectra.
	J. Fourier Anal. Appl. 12 (2006), no. 5, 483-494.
	
	\bibitem{FR}  Farkas, B.,   R\'ev\'esz, S., Tiles with no spectra in dimension 4. Math. Scand. 98 (2006), no. 1, 44-52.
		
	\bibitem{F}  Fuglede, B., Commuting self-adjoint partial differential operators and a group theoretic problem. J. Func. Anal. 16 (1974), 101-121.
	
	\bibitem{GeL}  Grepstad, S.,  Lev, N.,  Multi-tiling and Riesz bases. Adv. Math. 252 (2014), 1-6.
		
	\bibitem{H} Helson, H., Lectures on invariant subspaces. Academic Press, New York, 1964.
	
	\bibitem{Ios} Iosevich, A., Fuglede conjecture for lattices, preprint available at www.math.rochester.edu/people/faculty/iosevich/expository/FugledeLattice.pdf.
	
	\bibitem{IKT1}  Iosevich, A.,  Katz N.,  Tao T., The Fuglede spectral conjecture holds for convex bodies in the plane.
	Math. Res. Lett. 10 (2003), 559-570.
		
	\bibitem{KN}  Kozma, G., Nitzan, S., Combining Riesz bases. Invent. Math 199 (2015) no. 1 267-285 .
	
	\bibitem{KNd}  Kozma, G., Nitzan, S., Combining Riesz bases in $\R^d.$ Revista matem\'atica iberoamericana, ISSN 0213-2230, Vol. 32, Nº 4, 2016, p\'ags. 1393-1406.
		
	\bibitem{Ko}  Kolountzakis, M., Multiple lattice tiles and Riesz bases of exponentials. Proc. Amer. Math. Soc. 143 (2015), 
	no. 2, 741-747
		
	\bibitem{KM} Kolountzakis, M., Matolcsi M., Tiles with no spectra. Forum Math. 18 (2006), no. 3, 519-528.
	
	\bibitem{Ko1}  Kolountzakis, M., Non-symmetric convex domains have no basis of exponentials. Ilinois J. Math. 44 (2000), no.3, 542-550.
	
	\bibitem{MM1} Matei, B.,  Meyer, Y.: Simple quasicrystals are sets of stable sampling, Complex Var. Elliptic Equ.  {\bf{55}} (2010) 947-964. 
	
	\bibitem{MM2} Matei, B., Meyer, Y.: Quasicrystals are sets of stable sampling, C. R. Math. Acad. Sci. Paris  {\bf{346}} (2008) 1235-1238. 
	
	\bibitem{NOU16} Nitzan, S., Olevskii, A. and Ulanovskii, A. : Exponential frames on unbounded sets, Proc. Amer. Math. Soc. 144 (2016), 
	no. 1, 109-118.
	
	
	\bibitem{S} Seip, K., Interpolation and Sampling in Spaces of Analytic Functions, University Lecture Series, 33, American Mathematical Society, Providence, RI (2004), xii+139 pp.
	
	\bibitem{T} Tao, T., Fuglede's conjecture is false in 5 and higher dimensions. Math. Res. Lett. 11 (2004), no. 2-3, 251-258.
	
	\bibitem{T2} Tao, T., An uncertainty Principle For Cyclic Groups of Prime Order. Math. Res. Lett. 12 (2005), 121-127.
	
	\bibitem{Y} Young, R.M., An introduction to nonharmonic Fourier series. Revised first edition. Academic Press, Inc., San Diego, CA, 2001.
	
\end{thebibliography}
\end{document}